\documentclass{amsart}

\newtheorem{theorem}{Theorem}
\newtheorem{lemma}[theorem]{Lemma}
\newtheorem{proposition}[theorem]{Proposition}
\newtheorem{corollary}[theorem]{Corollary}

\newtheorem{conjecture}[theorem]{Conjecture}

\theoremstyle{definition}
\newtheorem{definition}[theorem]{Definition}
\newtheorem{example}[theorem]{Example}

\usepackage{amscd,amssymb}

\begin{document}

\title[Algebraic properties of codimension series]
{Algebraic properties of codimension series\\
of PI-algebras}

\author[S. Boumova, V. Drensky]
{Silvia Boumova, Vesselin Drensky}
\address{Silvia Boumova: Higher School of Civil Engineering ``Lyuben Karavelov'',
175 Suhodolska Str., 1373 Sofia, Bulgaria, and
Institute of Mathematics and Informatics,
Bulgarian Academy of Sciences,
1113 Sofia, Bulgaria}
\email{silvi@math.bas.bg}
\address{Vesselin Drensky: Institute of Mathematics and Informatics,
Bulgarian Academy of Sciences,
1113 Sofia, Bulgaria}
\email{drensky@math.bas.bg}


\begin{abstract}
Let $c_n(R)$, $n=0,1,2,\ldots$, be the codimension sequence of the PI-algebra $R$ over a field of characteristic 0
with T-ideal $T(R)$ and let $c(R,t)=c_0(R)+c_1(R)t+c_2(R)t^2+\cdots$ be the
codimension series of $R$ (i.e., the generating function of the codimension sequence of $R$).
Let $R_1,R_2$ and $R$ be PI-algebras such that $T(R)=T(R_1)T(R_2)$.
We show that if $c(R_1,t)$ and $c(R_2,t)$ are rational functions,
then $c(R,t)$ is also rational. If $c(R_1,t)$ is rational and $c(R_2,t)$ is algebraic,
then $c(R,t)$ is also algebraic.
The proof is based on the fact that the product of two exponential generating functions behaves as the exponential generating function of
the sequence of the degrees of the outer tensor products of two sequences of representations of the symmetric groups $S_n$.
\end{abstract}

\maketitle

\section*{Introduction}

With every infinite sequence of complex numbers  $a_0,a_1,a_2,\ldots$ or, shortly, $\{a_n\}_0^{\infty}$,
or even $\{a_n\}$,
we associate two formal power series: {\it the ordinary generating function}
\[
a(t)=\sum_{n\geq 0}a_nt^n
\]
and {\it the exponential generating function}
\[
\widetilde{a}(t)=\sum_{n\geq 0}a_n\frac{t^n}{n!}.
\]
In this way we have three $\mathbb C$-vector spaces:
\begin{enumerate}
\item[$\bullet$] The space $\mathcal S$ of all sequences $\{a_n\}$;
\item[$\bullet$] The space $\mathcal G$ of all generating functions $a(t)$;
\item[$\bullet$] The space $\mathcal E$ of all exponential generating functions $\widetilde{a}(t)$.
\end{enumerate}
Clearly, $\mathcal G$ and $\mathcal E$ coincide with the vector space ${\mathbb C}[[t]]$ of formal power series.

Let $R$ be a unital $K$-algebra with polynomial identity (or a PI-algebra), where $K$ is a fixed field of characteristic 0.
Let $T(R)\subset K\langle X\rangle$ be the T-ideal of $R$, where
$K\langle X\rangle=K\langle x_1,x_2,\ldots\rangle$ is the free unital algebra of countable rank over $K$.
Maybe the most important object in the quantitative study of PI-algebras is the codimension sequence $\{c_n(R)\}$, where
\[
c_n(R)=\dim(P_n/(P_n\cap T(R))),\quad n=0,1,2,\ldots.
\]
Here $P_n$ is the vector subspace of $K\langle X\rangle$ spanned by all multilinear monomials
$x_{\sigma(1)}\cdots x_{\sigma(n)}$ of degree $n$ ($\sigma$ is in the symmetric group $S_n$).
See the book by Giambruno and Zaicev \cite{GZ} for a background on PI-algebras and the properties of their codimensions.
By the classical theorem of Regev \cite{R0} the codimension sequence is exponentially bounded. Hence the radius of convergency
of {\it the codimension series of} $R$ (i.e., of the generating function of the codimension sequence $\{c_n(R)\}$)
\[
c(R,t)=\sum_{n\geq 0}c_n(R)t^n
\]
is
\[
r(R)=\frac{1}{\limsup_{n\to\infty}\sqrt[n]{c_n(R)}}
\]
and gives information for the growth of the codimensions of $R$. Later Giambruno and Zaicev
(See their book \cite{GZ} for an account) showed that the limit
\[
\exp(R)=\lim_{n\to\infty}\sqrt[n]{c_n(R)}
\]
exists and is an integer called {\it the exponent of} $R$. They also described the minimal with respect to the exponent
varieties of algebras (equivalently, the maximal T-ideals $T(R)$) of a given exponent. In this way they confirmed
a conjecture of Drensky \cite{D1} that the maximal with respect to the exponent T-ideals
are the products of the T-prime ideals introduced by Kemer in his structure theory of T-ideals, see his book \cite{K}.
Recently Berele and Regev \cite{BR2} for finitely generated algebras and Berele \cite{B}
in the general case confirmed (for unital algebras $R$) the conjecture of Regev
that for suitable $a\in{\mathbb R}$,$k\in{\mathbb Z}$, $b\in {\mathbb N}$
\[
c_n(R)\simeq an^{k/2}b^n,\quad n=0,1,2,\ldots.
\]

Regev, see his survey article \cite{R1}, determined for any $k>1$ the asymptotic behaviour of the codimension sequence
of the algebra $M_k(K)$ of $k\times k$ matrices with entries from $K$ and showed that
\[
c_n(M_k(K))\approx (2\pi)^{(1-k)/2}2^{(1-k^2)/2}
1!2!\ldots (k-1)!k^{(k^2+4)/2}n^{(1-k^2)/2}k^{2n+2}.
\]
Combining this with an old result of Jungen \cite{J} on the asymptotic behaviour of the coefficients of
algebraic power series, Beckner and Regev \cite{BcR} showed that for $n>1$ odd the codimension series $c(M_k(K),t)$ is not algebraic
over the field of rational functions ${\mathbb C}(t)$.
For $k=2$ the series $c(M_2(K),t)$ is algebraic and the conjecture is that it is not algebraic
for all even $k\geq 4$.

The starting point of our project was the following question of Regev (private communication):
{\it For a given PI-algebra $R$, is the codimension series $c(R,t)$ algebraic?}
The exact values of the codimension sequences are known for very few algebras, among them
the field $K$, the Grassmann algebra $E$, the algebra $M_2(K)$ and the tensor square $E\times_KE$ of the Grassmann algebra,
the algebras of $k\times k$ upper triangular matrices $U_k(K)$ and $U_k(E)$ with entries from the field and from the Grassmann algebra,
respectively. The rationality of of the codimension series holds for $K$, $E$, $U_k(K)$, $U_k(E)$
and algebras with polynomial growth of the codimension sequence. It is also known that the codimension series is algebraic for
$M_2(K)$ and $E\otimes E$.

In the present paper we study the rationality and algebraicity of the codimension series of products of T-ideals. We prove that
{\it if $R_1,R_2$ and $R$ are PI-algebras such that $T(R)=T(R_1)T(R_2)$ and $c(R_1,t), c(R_2,t)$ are rational, then $c(R,t)$ is also
rational. If $c(R_1,t)$ is rational and $c(R_2,t)$ is algebraic, then $c(R,t)$ is also algebraic.}
Formanek \cite{F} expressed the Hilbert series of the product of two T-ideals
in terms of the Hilbert series of the factors.
Berele and Regev \cite{BR1} translated this result in the language of cocharacters.
If $\{\chi_n(R_1)\}$ and $\{\chi_n(R_2)\}$ are, respectively, the cocharacter sequences
of the algebras $R_1$ and $R_2$, then
the cocharacter sequence of the T-ideal $T(R)=T(R_1)T(R_2)$ is
\[
\chi_n(R)=\chi_n(R_1)+\chi_n(R_2)+\chi_{(1)}\widehat\otimes \sum_{j=0}^{n-1}\chi_j(R_1)\widehat\otimes\chi_{n-j-1}(R_2)
-\sum_{j=0}^n\chi_j(R_1)\widehat\otimes\chi_{n-j}(R_2),
\]
$n=0,1,2,\ldots$, where $\widehat\otimes$ denotes the ``outer'' tensor product of characters of symmetric groups.
For irreducible characters it corresponds
to the Littlewood-Richardson rule for products of Schur functions:
\[
\chi_{\lambda}\widehat\otimes\chi_{\mu}=\sum_{\nu\vdash\vert\lambda\vert+\vert\mu\vert}c_{\lambda\mu}^{\nu}\chi_{\nu},
\]
where
\[
S_{\lambda}(t_1,\ldots,t_d)S_{\mu}(t_1,\ldots,t_d)
=\sum_{\nu\vdash\vert\lambda\vert+\vert\mu\vert}c_{\lambda\mu}^{\nu}S_{\nu}(t_1,\ldots,t_d),
\]
\[
\lambda=(\lambda_1,\ldots,\lambda_p),\quad \mu=(\mu_1,\ldots,\mu_q),\quad d\geq p+q,
\]
see e.g. the book by Macdonald \cite{M} for the rule. There is a simple formula for the exponential
codimension series of the product of two T-ideals $T(R)=T(R_1)T(R_2)$, see Drensky \cite{D1} and Petrogradsky \cite{P2}:
\[
\widetilde{c}(R,t)=\widetilde{c}(R_1,t)+\widetilde{c}(R_2,t)+(e^t-1)\widetilde{c}(R_1,t)\widetilde{c}(R_2,t).
\]
Exponential codimension series appeared also under the name {\it complexity functions}
in the work of Razmyslov, see his book \cite{Ra}, and Petrogradsky \cite{P1}
in their study of the codimension growth of the polynomial identities of Lie algebras.

In the present paper we know the behaviour of the product of two exponential generating functions and want to derive
algebraic properties of the corresponding ordinary generating function. Our main results are consequences of the following
more general ones.
{\it Let $\{a_n\}$ and $\{b_n\}$ be two sequences and let $\{c_n\}$ be the sequence
determined by the property $\widetilde{c}(t)=\widetilde{a}(t)\widetilde{b}(t)$. If the generating functions
$a(t)$ and $b(t)$ are rational, then $c(t)$ is also rational. If $a(t)$ is rational and $b(t)$ is algebraic, then $c(t)$ is also
algebraic.} The proof of the first fact is an easy consequence of well known properties of generating functions.
The proof of the second fact is more complicated.

\section{Properties of generating functions}

For the basic properties of generating functions see the books by Wilf \cite{Wi} or Lando \cite{L}.
Studying the codimension series of a product of T-ideals,
we may formalize the problem in the following way.
The vector spaces $\mathcal S$, $\mathcal G$ and $\mathcal E$ have natural structures of $\mathbb C$-algebras:
\begin{itemize}
\item[$\bullet$] The componentwise multiplication of sequences in $\mathcal S$:
\[
\{a_n\}\ast_{\mathcal S}\{b_n\}=\{a_nb_n\};
\]
\item[$\bullet$] The usual multiplication (using the Cauchy product rule) of formal power series in $\mathcal G$ and $\mathcal E$:
\[
\left(\sum_{n\geq 0}a_nt^n\right)\ast_{\mathcal G}\left(\sum_{n\geq 0}b_nt^n\right)
=\left(\sum_{n\geq 0}a_nt^n\right)\left(\sum_{n\geq 0}b_nt^n\right)=\sum_{n\geq 0}\left(\sum_{k=0}^na_kb_{n-k}\right)t^n,
\]
\[
\left(\sum_{n\geq 0}a_n\frac{t^n}{n!}\right)\ast_{\mathcal G}\left(\sum_{n\geq 0}b_n\frac{t^n}{n!}\right)
=\left(\sum_{n\geq 0}a_n\frac{t^n}{n!}\right)\left(\sum_{n\geq 0}b_n\frac{t^n}{n!}\right)
\]
\[
=\sum_{n\geq 0}\left(\sum_{k=0}^n\binom{n}{k}a_kb_{n-k}\right)\frac{t^n}{n!}.
\]
\end{itemize}
The maps $\gamma:{\mathcal S}\to {\mathcal G}$ and $\varepsilon:{\mathcal S}\to {\mathcal E}$ defined by
\[
\gamma:\{a_n\}\to a(t)=\sum_{n\geq 0}a_nt^n,
\quad \varepsilon:\{a_n\}\to \widetilde{a}(t)=\sum_{n\geq 0}a_n\frac{t^n}{n!}
\]
are isomorphisms of vector spaces but not isomorphisms of algebras.
{\it If $\{a_n\}$ and $\{b_n\}$ satisfy some property, what can one say
about the properties of $\gamma(\{a_n\}\ast_{\mathcal S}\{b_n\})$ and
$\varepsilon(\{a_n\}\ast_{\mathcal S}\{b_n\})$? One may ask similar questions for
the properties of $\gamma^{-1}(a(t)\ast_{\mathcal S}b(t))$, $\varepsilon\circ\gamma^{-1}(a(t)\ast_{\mathcal S}b(t))$
and $\varepsilon^{-1}(\widetilde{a}(t)\ast_{\mathcal E}\widetilde{b}(t))$,
$\gamma\circ\varepsilon^{-1}(\widetilde{a}(t)\ast_{\mathcal E}\widetilde{b}(t))$.}

The following well known lemma gives the description of sequences with rational generating functions.

\begin{lemma}\label{rational generating function}
The following conditions for the sequence $\{a_n\}$ are equivalent:
\begin{itemize}
\item[(i)] The sequence $\{a_n\}$ satisfies a linear recurrence relation with constant coefficients $c_1,\dots,c_k$,
i.e., the elements  $a_0, a_1, \dots, a_{k-1}$ are arbitrary and
\[
a_{n+k} = c_1 a_{n+k-1} + c_2 a_{n+k-2} + \cdots + c_k a_n,\quad n=0,1,2,\ldots;
\]
\item[(ii)] There exist constants $\alpha_1,\ldots,\alpha_i$, polynomials $p_1(n),\ldots,p_i(n)$
and a nonnegative integer $n_0$ such that
for $n\geq n_0$ the elements $a_n$ are of the form
\[
a_n=p_1(n)\alpha_1^n+\cdots+p_i(n)\alpha_i^n;
\]
\item[(iii)] The generating function $a(t)$ is a rational function;
\item[(iv)] The exponential generating function $\widetilde{a}(t)$ satisfies a linear homogeneous differential
equation with constant coefficients;
\item[(v)] There exist constants $\beta_1,\ldots,\beta_i$ and polynomials $f_1(t),\ldots,f_i(t)$
such that the exponential generating function $\widetilde{a}(t)$ is of the form
\[
\widetilde{a}(t)=f_1(t)e^{\beta_1t}+\cdots+f_i(t)e^{\beta_it}.
\]
\end{itemize}
\end{lemma}

The equivalence between (ii) and (iii) immediately implies the following property of the Hadamard product
\[
a(t)\ast_Hb(t)=\gamma(\{a_n\}\ast_{\mathcal S}\{b_n\})
\]
of $a(t)$ and $b(t)$:
If $a(t)$ and $b(t)$ are rational, then $a(t)\ast_Hb(t)$
is also rational. A result of Jungen \cite{J} gives that if $a(t)$ is rational and $b(t)$ is algebraic, then
$a(t)\ast_Hb(t)$ is also algebraic.

There is also another point of view to the cocharacter sequence of the product of two T-ideals
which comes from the formula of Berele and Regev \cite{BR1} for the cocharacters of $R$ with $T(R)=T(R_1)T(R_2)$.
Let $\xi=\{\xi_n\}$ and $\eta=\{\eta_n\}$
be two sequences of $S_n$-characters with degrees $d_n(\xi)=d(\xi_n)$ and $d_n(\eta)=d(\eta_n)$, $n=0,1,2,\ldots$, respectively.
We embed $S_k$ and $S_{n-k}$ into $S_n$ assuming that $S_k$ and $S_{n-k}$ act, respectively, on $\{1,\ldots,k\}$ and
$\{k+1,\ldots,n\}$. Then we consider the $S_k\times S_{n-k}$ character $\xi_k\otimes \eta_{n-k}$. The outer product
$\xi_k\widehat{\otimes} \eta_{n-k}$ is defined as the induced $S_n$-character of $\xi_k\otimes \eta_{n-k}$ and its degree
is
\[
d(\xi_k\widehat{\otimes} \eta_{n-k})=\binom{n}{k}d_k(\xi)d_{n-k}(\eta).
\]
Hence,
\[
d(\xi\widehat{\otimes}\eta)=\left\{d_n(\xi\widehat{\otimes}\eta)=\sum_{k=0}^n\binom{n}{k}d_k(\xi)d_{n-k}(\eta)\right\}.
\]
In the special case when $\xi=\{\xi_n=\chi_{(n)}\}$ is the sequence of trivial $S_n$-characters, the outer product
$\xi\widehat{\otimes}\eta$ can be calculated by the Young rule. Regev \cite{R2} called the resulting sequence
$\xi\widehat{\otimes}\eta$ {\it Young derived}. In the general case, since $\xi\widehat{\otimes}\eta$ can be computed
by the Littlewood-Richardson rule, we call the sequence $d(\xi\widehat{\otimes}\eta)$
{\it Littlewood-Richardson derived} or {\it RL-derived}.
We transfer this definition to arbitrary sequences.

\begin{definition}\label{LR-derived sequences}
Let $\{a_n\}$ and $\{b_n\}$ be two sequences. Their {\it Littlewood-Richardson derived} (or {\it RL-derived}) {\it sequence}
$\{c_n\}=\{a_n\}\ast_{LR}\{b_n\}$ is defined by
\[
c_n=\sum_{k=0}^n\binom{n}{k}a_kb_{n-k},\quad n=0,1,2,\ldots.
\]
In other words,
\[
\{a_n\}\ast_{LR}\{b_n\}=\varepsilon^{-1}\left(\widetilde{a}(t)\ast_{\mathcal E}\widetilde{b}(t)\right).
\]
Similarly, the {\it Littlewood-Richardson derived generating function of the generating functions $a(t)$ and $b(t)$} is
\[
c(t)=a(t)\ast_{LR}b(t)=\gamma\circ\varepsilon^{-1}\left(\widetilde{a}(t)\ast_{\mathcal E}\widetilde{b}(t)\right)
=\sum_{n\geq 0}\left(\sum_{k=0}^n\binom{n}{k}a_kb_{n-k}\right)t^n.
\]
\end{definition}

\begin{proposition}\label{LR-rational}
If the generating functions $a(t)$ and $b(t)$ are rational, then their LR-derived $a(t)\ast_{LR}b(t)$ is also rational.
\end{proposition}

\begin{proof}
If $a(t)$ and $b(t)$ are rational, then by Lemma \ref{rational generating function} (v) the corresponding exponential
generating functions $\widetilde{a}(t)$ and $\widetilde{b}(t)$ are sums of expressions of the form
$f(t)e^{\alpha t}$ and $g(t)e^{\beta t}$, respectively, where $f(t)$ and $g(t)$ are polynomials and $\alpha,\beta\in\mathbb C$.
Hence $\widetilde{c}(t)=\widetilde{a}(t)\ast_{\mathcal E}\widetilde{b}(t)$ is a sum of expressions of the form
$f(t)g(t)e^{(\alpha+\beta)t}$. By Lemma \ref{rational generating function} (iii)
$c(t)=a(t)\ast_{LR}b(t)$ is also a rational function.
\end{proof}

The following lemmas express the LR-derived of two generating functions, if one of them is of special kind.

\begin{lemma}\label{LR-derived of fraction}
If $\alpha$ is a nonzero constant and $b(t)$ is a generating function, then
\[
\frac{1}{1-\alpha t}\ast_{LR}b(t) = \frac{1}{1-\alpha t}~b \left( \frac{t}{1-\alpha t} \right).
\]
\end{lemma}

\begin{proof} The proof is similar to the proof for the relation between the ordinary and proper codimension series
of PI-algebras given in \cite{D0}:
\begin{center}
\begin{tabular}{ccccc}
$ \displaystyle \frac{1}{1-\alpha t}\ast_{LR}b(t) $
& = & $\displaystyle \left( \sum_{n\geq 0} \alpha^nt^n\right)\ast_{LR} \left( \sum_{n\geq 0} b_nt^n\right) $
& = & $ \displaystyle \sum_{n\geq 0} \sum_{k=0}^n {n\choose k} \alpha^kt^kb_{n-k}t^{n-k} $ \\ [20pt]

& = & $ \displaystyle \sum_{n\geq 0} \sum_{k\geq 0} {n+k\choose k} \alpha^kt^k b_{n}t^{n}$
& = & $ \displaystyle \sum_{n\geq 0}b_nt^n \sum_{k\geq 0}{k+n \choose n} \alpha^k t^k   $ \\ [20pt]

& = & $ \displaystyle  \sum_{n\geq 0}b_nt^n \frac{1}{(1-\alpha t)^{n+1}} $
& = & $ \displaystyle \frac{1}{1-\alpha t} \sum_{n\geq 0} b_n\frac{t^n}{(1-\alpha t)^n} $ \\ [20pt]

& = & $ \displaystyle  \frac{1}{1-\alpha t}  b\left(\frac{t}{1-\alpha t} \right)$.
&&
\end{tabular}
\end{center}
\end{proof}

\begin{lemma}\label{LR-derived of monomial}
If $p$ is a positive integer and $b(t)$ is a generating function, then
\[
t^p\ast_{LR}b(t) = \frac{t^pd^p b(t)}{p!dt^p}.
\]
\end{lemma}

\begin{proof}
Let $\delta_{pk}$ be the Kronecker symbol. We write consequently
\begin{center}
\begin{tabular}{lll}
$ \displaystyle t^p\ast_{LR}b(t)$ & = & $ \displaystyle \left(\sum_{n\geq 0}\delta_{pn}t^n\right)\ast_{LR}b(t)$\\ [20pt]
                          & = & $\displaystyle \sum_{n \geq 0} \sum_{k=0}^n {n \choose k} \delta_{pk} t^k b_{n-k} t^{n-k}
                                 = \sum_{n \geq 0} {n \choose p} b_{n-p} t^n  $\\ [20pt]
                          & = & $ \displaystyle \frac{t^p}{p!} \sum_{n \geq p} n(n-1)\cdots (n-p+1) b_n t^{n-p}  $ \\ [20pt]
                          & = & $ \displaystyle \frac{x^pd^p}{p!dt^p}b(t)$.
\end{tabular}
\end{center}
\end{proof}

\begin{lemma}\label{power of t times fraction}
The generating function $\displaystyle  \frac{1}{(1-\beta t)^p}$, $\beta\not=0$,  is a linear combination of
\[
\frac{1}{1-\beta t}\ast_{LR} t^q,\quad q=0,1,\dots,p.
\]
\end{lemma}

\begin{proof}

We present $t^q$ as a polynomial in $1-\beta t$:
\[
t^q=\alpha_0+\alpha_1(1-\beta t)+\cdots+\alpha_{q-1}(1-\beta t)^{q-1}+\alpha_q(1-\beta t)^q,\quad \alpha_i\in\mathbb C,
\]
\[
\frac{t^q}{(1-\beta t)^q}=\frac{\alpha_0}{(1-\beta t)^q}+\frac{\alpha_1}{(1-\beta t)^{q-1}}
+\cdots+\frac{\alpha_{q-1}}{(1-\beta t)^{q-1}}+\frac{\alpha_q}{(1-\beta t)^q}.
\]
For $t=\beta$ we obtain $\alpha_0=\beta^q\not=0$. In this way
$\displaystyle \frac{1}{(1-\beta t)^p}$  can be expressed as a linear combination of
\[
\frac{t^q}{(1-\beta t)^q},\quad q=0,1,\dots,p.
\]
This completes the proof because by Lemma \ref{LR-derived of monomial}
\[
\frac{t^q}{q!(1-\beta t)^q}= \frac{1}{\beta^q}\frac{t^qd^q}{q!dt^q}\frac{1}{1-\beta t}
=\frac{1}{\beta^q}\left(t^q\ast_{LR}\frac{1}{1-\beta t}\right).
\]
\end{proof}

\begin{theorem}\label{rational plus algebraic}
If $a(t)$ and $b(t)$ are generating functions, $a(t)$ is rational and $b(t)$ is algebraic over ${\mathbb C}(t)$,
then LR-derived generating function $a(t)\ast_{LR}b(t)$ is algebraic.
\end{theorem}

\begin{proof}
Since $a(t)\in{\mathbb C}[[t]]$ is rational, it is a sum of a polynomial and fractions of the form
$\displaystyle \frac{\alpha}{(1-\beta t )^p}$, $\alpha,\beta\in {\mathbb C}$, and $\beta$ is nonzero.
Linear combinations, products and fractions of algebraic functions are also algebraic.
Hence it is sufficient to consider the cases
$a(t)=t^r$ and $\displaystyle a(t)=\frac{1}{(1-\beta t)^p}$. By Lemma \ref{power of t times fraction}
the latter can be replaced by $\frac{1}{1-\beta t}\ast_{LR}t^p$.
Since $b(t)$ is algebraic, it satisfies
\[
f(t,b(t))=b^m(t) + c_1(t)b^{m-1}(t) + \cdots + c_{m-1}(t)b(t) + c_m(t)=0
\]
for some $c_j(t) \in {\mathbb C}(t)$ and we choose the equation of minimal degree $m$.
Applying $\displaystyle \frac{d}{dt}$ we obtain
\[
0=\frac{df(t,b(t))}{dt}=\frac{db(t)}{dt}(mb^{m-1}(t)+(m-1)c_1(t)b^{m-2}(t) + \cdots + c_{m-1}(t))
\]
\[
+\sum_{i=1}^m\frac{dc_i(t)}{dt}b^{m-i}(t)=q(t,b(t))\frac{db(t)}{dt}+p(t,b(t)),
\]
where $p(t,b(t)),q(t,b(t))$ are polynomials in $b(t)$ with coefficients in ${\mathbb C}(t)$. Clearly $\deg_bq(t,b(t))<m$
and by the minimality of the degree $m$ of $f(t,b(t))$, we obtain that $q(t,b(t))\not=0$ in ${\mathbb C}(t)$ and
\[
\frac{db(t)}{dt}=-\frac{p(t,b(t))}{q(t,b(t))}
\]
is algebraic. Hence $t^p\ast_{LR}b(t)$ is also algebraic by Lemma \ref{LR-derived of monomial}.
Since
\[
b^m\left(\frac{t}{1-\beta t} \right) + c_1\left(\frac{t}{1-\beta t} \right)b^{m-1}\left(\frac{t}{1-\beta t} \right) + \cdots
+ c_m\left(\frac{t}{1-\beta t} \right)=0,
\]
$\displaystyle b\left(\frac{t}{1-\beta t} \right)$ is algebraic and by Lemma \ref{LR-derived of fraction}
the same holds for $\displaystyle \frac{1}{1-\alpha t}\ast_{LR}b(t)$. Finally, the operation $\ast_{LR}$ is associative
because is the image in $\mathcal G$ of the multiplication in $\mathcal E$. Hence by Lemma \ref{power of t times fraction}
\[
\left(\frac{1}{1-\beta t}\ast_{LR} t^p\right)\ast_{LR}b(t)=\frac{1}{1-\beta t}\ast_{LR} \left(t^p\ast_{LR}b(t)\right)
\]
which is algebraic again.
\end{proof}

\section{Codimension series}

Let $K$ be a field of characteristics $0$ and let $K\langle X\rangle$ be the free associative algebra of countable rank.
Recall that the element $f(x_1,\ldots,x_d)\in K\langle X\rangle$ is a {\it polynomial identity for the algebra} $R$
if $f(r_1,\ldots,r_d)=0$ for all $r_1,\ldots,r_d\in R$. If $f$ is a nonzero element of $K\langle X\rangle$, then $R$ is called an
{\it algebra with polynomial identity} or a{\it PI-algebra}. The set $T(R)$ of all polynomial identities for $R$ is called
the {\it T-ideal of} $R$. Let
\[
P_n=\text{span}\{x_{\sigma(1)}\cdots x_{\sigma(n)}\mid \sigma\in S_n\}\subset K\langle X\rangle
\]
be the vector space of all {\it multilinear polynomials} of degree $n$. Since $T(R)$ is generated as a T-ideal
by its multilinear identities, this motivates their intensive study. In particular,
the codimension sequence of $R$ is
\[
c_n(R)=\dim(P_n/(P_n\cap T(R))),\quad n=0,1,2,\ldots.
\]
We shall consider also the ordinary and exponential codimension series
\[
c(R,t)=\sum_{n\geq 0}c_n(R)t^n,\quad \widetilde{c}(R,t)=\sum_{n\geq 0}c_n(R)\frac{t^n}{n!}.
\]
\begin{theorem}\label{rational and algebraic codimension series}
If $R_1,R_2$ and $R$ are PI-algebras such that $T(R)=T(R_1)T(R_2)$ and the codimension series $c(R_1,t), c(R_2,t)$ are rational,
then the codimension series $c(R,t)$ is also rational. If $c(R_1,t)$ is rational and $c(R_2,t)$ is algebraic, then $c(R,t)$ is also algebraic.
\end{theorem}

\begin{proof}
The ordinary and the exponential generating functions of the sequence $\{a_n=1\}_0^{\infty}$ are, respectively,
\[
a(t)=\frac{1}{1-t},\quad \widetilde{a}(t)=e^t.
\]
Translating the formula for the exponential codimension series
\[
\widetilde{c}(R,t)=\widetilde{c}(R_1,t)+\widetilde{c}(R_2,t)+(e^t-1)\widetilde{c}(R_1,t)\widetilde{c}(R_2,t)
\]
in the language of ordinary codimension series, we obtain
\[
c(R,t)=c(R_1,t)+c(R_2,t)+\left(\frac{1}{1-t}-1\right)\ast_{LR}c(R_1,t)\ast_{LR}c(R_2,t).
\]
Now the proof follows immediately from Proposition \ref{LR-rational} and Theorem \ref{rational plus algebraic}.
\end{proof}

Below we give examples of algebras with rational or algebraic codimension series.
We shall use without reference, see \cite{D0} or \cite{D3}, that the ordinary codimension sequence $\{c_n(R)\}$
and the so called {\it proper codimension sequence} $\{\gamma_n(R)\}$ of $R$
(and the corresponding ordinary and exponential series) are related by
\[
c_n(R)=\sum_{k=0}^{n_0}\binom{n}{k}\gamma_k(R),\quad \widetilde{c}(R,t)=e^t\widetilde{\gamma}(R,t),
\]
\[
c(R,t)=\frac{1}{1-t}\gamma\left(R,\frac{t}{1-t}\right)=\frac{1}{1-t}\ast_{LR}\gamma(R,t).
\]

\begin{example}\label{known codimension sequences}
(i) It is obvious that
\[
c_n(K)=1,\quad n=0,1,2,\ldots,\quad c(K,t)=\frac{1}{1-t},\quad \widetilde{c}(K,t)=e^t.
\]
(ii) It is also known, see \cite{D3} for an elementary proof, that
\[
c_0(E)=1,\quad c_n(E)=2^{n-1},\quad n=1,2,\ldots,\]
\[
c(E,t)=\frac{1}{2}\left(1+\frac{1}{1-2t}\right)=\frac{1-t}{1-2t},\quad
\widetilde{c}(E,t)=\frac{1}{2}\left(1+e^{2t}\right).
\]
\end{example}

\begin{example}\label{2x2 matrices}
The codimension sequence
and the codimension series of $M_2(K)$ are, see Procesi \cite{Pr}:
\[
c_n(M_2(K))={1\over n+2}{2n+2\choose n+1}-{n\choose 3}+1-2^n,\quad n=0,1,2,\ldots,
\]
\[
c(M_2(K),t)=\sum_{n\geq 0}c_n(M_2(K))t^n=
{1\over 2t^2}\left(1-2t-\sqrt{1-4t}\right)
\]
\[
-{t^3\over (1-t)^4}+{1\over 1-t}-{1\over 1-2t}.
\]
Another proof is given in \cite{D2}.
The translation of the approach in \cite{D2} gives
\[
\widetilde{c}(M_2(K),t)=e^t\left(e^t\sum_{k\geq 0}\deg(\chi_{(k^2)})\frac{t^{2k}}{(2k)!}-\frac{t^3}{3!}-\sum_{n\geq 1}\frac{t^n}{n!}\right),
\]
where $\deg(\chi_{(k^2)})$ is the degree of the irreducible $S_{2k}$-character $\chi_{(k^2)}$. By the hook formula
\[
\deg(\chi_{(k^2)})=\frac{(2k)!}{k!(k+1)!}=C_k,
\]
the $k$-th Catalan number. It is known, see e.g. \cite{SW}, that
\[
\sum_{k\geq 0}C_k\frac{t^{2k}}{(2k)!}=\frac{I_1(2t)}{t},
\]
where $I_n(z)$ is the {\it modified Bessel function of the first kind}. The function $I_n(z)$, see e.g. \cite{W},
can be defined by the contour integral
\[
I_n(z)=\frac{1}{2\pi i}\oint e^{(z/2)(\xi+1/\xi)}\xi^{-n-1}d\xi,
\]
where the contour encloses the origin and is traversed in a counterclockwise direction.
It can be also expressed in terms of $I_0(z)$ as
\[
I_n(z)=T_n\left(\frac{d}{dz}\right)I_0(z),
\]
where $T_n(z)$ is the Chebyshev polynomial of the first kind defined by the identity $T_n(\cos(\theta))=\cos(n\theta)$
and
\[
I_0(z)=\sum_{k\geq 0}\frac{\left(\frac{1}{4}z^2\right)^k}{(k!)^2}=\frac{1}{\pi}\int_0^{\pi}e^{z\cos(\theta)}d\theta.
\]
Since $T_1(z)=z$, we obtain that
\[
I_1(z)=\frac{dI_0(z)}{dz}.
\]
Hence
\[
\widetilde{c}(M_2(K),t)=e^{2t}\left(\frac{I_1(2t)}{t}-1\right)+e^t\left(1-\frac{t^3}{3!}\right).
\]
\end{example}

\begin{example}\label{E tensor E}
The codimension sequence and the codimension series of $E\otimes_KE$ were calculated in \cite{D2} as a translation
of the description of the polynomial identities of $E\otimes_KE$ given by Popov \cite{Po1}:
\[
c_0(E\otimes_K E) = 1,\quad c_n(E\otimes_K E) = \frac{1}{2}\binom{2n}{n} +n + 1 - 2^n,\quad n=1,2,\ldots,
\]
\[
c(E\otimes_K E,t) = \frac{1}{2} + \frac{1}{2\sqrt{1-4t}} +\frac{t}{(1-t)^2} + \frac{1}{1-t} - \frac{1}{1-2t}.
\]
As in Example \ref{2x2 matrices}, in order to obtain a closed formula for the exponential codimension series of $E\otimes_KE$
we need to express the exponential generating function
\[
\widetilde{g}(t)=\sum_{n\geq 0}\binom{2n}{n}\frac{t^n}{n!}
=\sum_{n\geq 0}\frac{1}{n+1}\binom{2n}{n}\frac{(n+1)t^n}{n!}
\]
\[
=\sum_{n\geq 0}C_n\frac{(n+1)t^n}{n!}
=\frac{d}{dt}\left(t\sum_{n\geq 0}C_n\frac{t^n}{n!}\right).
\]
According to \cite{SW},
\[
\sum_{n\geq 0}C_n\frac{t^n}{n!}=e^{2t}(I_0(2t)-I_1(2t)).
\]
Since the Chebyshev polynomial of the first kind $T_2(z)$ is equal to $2z^2-1$, we obtain
\[
I_2(z)=T_2(\frac{d}{dz})I_0(z)=2\frac{d^2I_0(z)}{dz^2}-I_0(z)=2\frac{dI_1(z)}{dz}-I_0(z),
\]
\[
\frac{dI_1(z)}{dz}=\frac{1}{2}(I_0(z)+I_2(z)),
\]
\[
\widetilde{g}(t)=\frac{d}{dt}(te^{2t}(I_0(2t)-I_1(2t)))=e^{2t}((1+t)I_0(2t)-I_1(2t)-tI_2(2t)).
\]
As in Example \ref{2x2 matrices}, the exponential generating function $\widetilde{c}(E\otimes_KE,t)$ involves
modified Bessel functions of the first kind:
\[
\widetilde{c}(E\otimes_KE,t)=1+\frac{1}{2}\sum_{n\geq 1}\binom{2n}{n}\frac{t^n}{n!}+\sum_{n\geq 1}n\frac{t^n}{n!}
+\sum_{n\geq 1}\frac{t^n}{n!}-\sum_{n\geq 1}2^n\frac{t^n}{n!}
\]
\[
=\frac{1}{2}+\frac{1}{2}\sum_{n\geq 0}\binom{2n}{n}\frac{t^n}{n!}+\sum_{n\geq 1}\frac{t^n}{(n-1)!}
+\left(\sum_{n\geq 0}\frac{t^n}{n!}-1\right)-\left(\sum_{n\geq 0}2^n\frac{t^n}{n!}-1\right)
\]
\[
=\frac{1}{2}e^{2t}((1+t)I_0(2t)-I_1(2t)-tI_2(2t))+\frac{1}{2}+(1+t)e^t-e^{2t}.
\]
\end{example}

\begin{example}\label{T-ideals generated by one polynomial identity}
Let $f\in K\langle X\rangle$ and let $(f)^T$ be the T-ideal generated by $f$. As in the case of a PI-algebra $R$
we define the codimension sequence of $(f)^T$ by
\[
c_n((f)^T)=\dim(P_n/(P_n\cap (f)^T)),\quad n=0,1,2,\ldots,
\]
and the ordinary and exponential codimension series $c((f)^T,t)$ and $\widetilde{c}((f)^T,t)$.

\noindent (i) Volichenko \cite{V} described the structure of the factor spaces $P_n/(P_n\cap(f_4)^T)$, $n=0,1,2,\ldots$,
for the T-ideal generated by the commutator of length 4
\[
f_4=[x_1,x_2,x_3,x_4]=[[[x_1,x_2],x_3],x_4].
\]
Since the Grassmann algebra $E$ satisfies the identity $f_4=0$, it is convenient to compare the codimensions of
$(f_4)^T$ with the codimensions of $E$. It follows from \cite{V} that
\[
c_n((f_4)^T)=c_n(E)+2\binom{n}{4},\quad n=0,1,2,\ldots,
\]
\[
c((f_4)^T,t)=c(E,t)+\frac{2t^4}{(1-t)^5},\quad \widetilde{c}((f_4)^T,t)=\widetilde{c}(E,t)+\frac{2t^4}{4!}e^t.
\]

\noindent (ii) Stoyanova-Venkova \cite{SV1, SV3} described the structure of $P_n/(P_n\cap(f_5)^T)$, $n=0,1,2,\ldots$,
for the T-ideal generated by the commutator of length 5. Her result gives the following description of the
so called {\it proper cocharacter sequence} $\xi_n((f_5)^T)$, $n=0,1,2,\ldots$:
\[
\xi_n((f_5)^T)=\begin{cases}
\chi_{(0)},\text{ if }n=0,\\
0,\text{ if }n=1,\\
\chi_{(n-1,1)},\text{ if }n=2,3,\\
\chi_{(3,2)}+\chi_{(3,1^2)}+\chi_{(2^2,1)}+\chi_{(2,1^3)},\text{ if }n=5,\\
\chi_{(3^2)}+\chi_{(3,1^3)}+\chi_{(2^2,4)}+\chi_{(2,1^4)}+\chi_{(1^6)},\text{ if }n=6,\\
\chi_{(3,1^{n-3})}+\chi_{(2^2,1^{n-4})}+\chi_{(2,1^{n-2})}+\chi_{(1^n)},\text{ if } n=4 \text{ or } n>6 \text{ even},\\
\chi_{(3,1^{n-3})}+\chi_{(2^2,1^{n-4})}+\chi_{(2,1^{n-2})},\text{ if }n>5 \text{ odd.}\\
\end{cases}
\]
This implies that
\[
\widetilde{c}((f_5)^T,t)=\widetilde{c}(E,t)+\left(\frac{2t^3}{3!}+\frac{5t^5}{5!}+\frac{5t^6}{6!}+\sum_{n\geq 4}\frac{n(n-2)t^n}{n!}\right)e^t
\]
\[
= \left(\frac{1}{2}-t+t^2\right)e^{2t}+\left(t-\frac{t^3}{3!}+\frac{5t^5}{5!}+\frac{5t^6}{6!}\right)e^t+\frac{1}{2}.
\]

\noindent (iii) Kemer \cite{K0} described the codimensions of the T-ideal generated by the standard polynomial of degree 4
\[
s_4=\sum_{\sigma\in S_4}\text{sign}(\sigma)x_{\sigma(1)}x_{\sigma(2)}x_{\sigma(3)}x_{\sigma(4)}.
\]
It follows, see \cite{D2}, that
\[
c_n((s_4)^T)=c_n(M_2(K))+5\binom{n}{5}+5\binom{n}{6},
\]
\[
c((s_4)^T,t)=c(M_2(K),t)+\frac{5t^5}{(1-t)^6}+\frac{5t^6}{(1-t)^7},
\]
\[
\widetilde{c}((s_4)^T,t)=\widetilde{c}(M_2(K),t)+\left(\frac{5t^5}{5!}+\frac{5t^6}{6!}\right)e^t.
\]

\noindent (iv) Popov \cite{Po2} described the cocharacters of the centre-by-metabelian polynomial
\[
f=[[x_1,x_2],[x_3,x_4],x_5]
\]
which is satisfied by $E\otimes_KE$. The formula for the codimensions of the T-ideal $(f)^T$ is, see \cite{D2},
\[
c_n((f)^T)=c_n(E\otimes_KE)+5\binom{n}{5}+5\binom{n}{6}.
\]
Hence, as in the case of the T-ideal generated by $s_4$,
\[
c((f)^T,t)=c(E\otimes_KE,t)+\frac{5t^5}{(1-t)^6}+\frac{5t^6}{(1-t)^7},
\]
\[
\widetilde{c}((f)^T,t)=\widetilde{c}(E\otimes_KE,t)+\left(\frac{5t^5}{5!}+\frac{5t^6}{6!}\right)e^t.
\]
\noindent (v) Nikolaev \cite{N} studied the T-ideal generated by the Hall (or Wagner) identity
\[
h=[[x_1,x_2]^2,x_3].
\]
The codimensions of $(h)^T$ were computed in \cite{D2}:
\[
c_n((h)^T)=c_n(M_2(K))+2^{n-1}-1-\binom{n}{2}+14\binom{n}{5}+33\binom{n}{6}+14\binom{n}{7},\quad n\geq 1.
\]
Again, we can express the ordinary and exponential codimension series of $(h)^T$ in by the corresponding series of $M_2(K)$.
\end{example}

\begin{example}\label{polynomial codimension}
Kemer \cite{K1, K2} described the algebras $R$ (also nonunital) with a polynomially
bounded codimension sequence, i.e., $c_n(R)\leq an^k$, $n=1,2,\ldots$, for a positive real $a$
and a nonnegative integer $k$. The following conditions are equivalent to the polynomial growth
of $\{c_n(R)\}_1^{\infty}$:

\noindent (i) The polynomial identities of $R$ are {\it of bounded colength}. If the cocharacter sequence of $R$ is
\[
\chi_n(R)=\chi_{S_n}(P_n/(P_n\cap T(R))=\sum_{\lambda\vdash n}m_{\lambda}(R)\chi_{\lambda},\quad n=1,2,\ldots,
\]
there is a constant $\ell$ such that
\[
\sum_{\lambda\vdash n}m_{\lambda}(R)\leq \ell,\quad n=1,2,\ldots;
\]
\noindent (ii) The algebra $R$ satisfies polynomial identities $f_1=0$ and $f_2=0$ such that
the Grassmann algebra $E$ does not satisfy the identity $f_1=0$ and the algebra $U_2(K)$
of the $2\times 2$ upper triangular matrices does not satisfy $f_2=0$.

\noindent By \cite{D4} the codimension series $c(R,t)$ is a rational function.
If we work with unital algebras, there are two more equivalent conditions:

\noindent (iii) For suitable $k$ and $m$ the algebra $R$ satisfies the Engel identity
\[
[x_1,\underbrace{x_2,\ldots,x_2}_{k\text{ times}}]=0
\]
and the standard identity
\[
s_m(x_1,\ldots,x_m)=\sum_{\sigma\in S_m}\text{sign}(\sigma)x_{\sigma(1)}\cdots x_{\sigma(m)}=0.
\]
\noindent (iv) The proper codimension sequence $\{\gamma_n(R)\}$ of $R$ becomes 0 for $n$ sufficiently large.

If $\gamma_n(R)=0$ for $n>n_0$, then
\[
c(R,t)=\frac{1}{1-t}\ast_{LR}\left(\sum_{k=0}^{n_0}\gamma_k(R)t^k\right),
\]
which is a rational function by Theorem \ref{LR-rational} (and is a partial case of the result in \cite{D4}).
\end{example}

An immediate consequence of Theorem \ref{rational and algebraic codimension series} and Examples \ref{known codimension sequences} --
\ref{polynomial codimension} is the following:

\begin{corollary}\label{products of examples}
Let $R_i$, $i=1,\ldots,k$, be some of the algebras $K$, $E$, algebras with T-ideals generated by the commutators of length $4$ and $5$
or algebras with polynomial growth of the codimensions. Then the codimension series $c(R,t)$ of the algebra $R$ with
$T(R)=T(R_1)\cdots T(R_k)$ is rational. If $R_0$ is one of the algebras $M_2(K)$, $E\otimes_KE$ or an algebra with T-ideal
generated by $s_4(x_1,x_2,x_3,x_4)$, the Hall polynomial $[[x_1,x_2]^2,x_3]$ or the centre-by-metabelian polynomial
$[[x_1,x_2],[x_3,x_4],x_5]$, then the codimension series of the algebra $R$ with
\[
T(R)=T(R_1)\cdots T(R_i)T(R_0)T(R_{i+1})\cdots T(R_k)
\]
is algebraic for every $i=0,1,\ldots,k$.
\end{corollary}

Partial cases of the corollary follow from other results, e.g. Drensky \cite{D1} and Petrogradsky \cite{P2}
for $U_k(K)$ and $U_k(E)$, Stoyanova-Venkova \cite{SV2} and Centrone \cite{C} for the T-ideal generated
by $[x_1,x_2,x_3][x_4,x_5]$.

Completing the conjecture of Regev that for even $k\geq 4$ the codimension series $c(M_k(K),t)$ is not algebraic, we have the following:

\begin{conjecture}
Let $R$ be a PI-algebra with $T(R)=T(R_1)\cdots T(R_k)$, where $R_i=K,E,M_2(K),E\otimes_KE$ and at least two of the algebras $R_i$
are equal to $M_2(K)$ or $E\otimes_KE$. Then the codimension series $c(R,t)$ is not algebraic.
\end{conjecture}

\section*{Acknowledgements}

The authors are very grateful to Amitai Regev for the stimulating discussions.

\end{document}